\documentclass[12pt]{amsart}
\usepackage{amsmath}
\usepackage{amsthm}
\usepackage{amsfonts} 
\usepackage{hyperref}
\usepackage{epsfig}

\numberwithin{equation}{section}

\allowdisplaybreaks

\newtheorem{theorem}{Theorem}[section]
\newtheorem{proposition}{Proposition}[section]
\newtheorem{lemma}{Lemma}[section]
\newtheorem{remark}{Remark}[section]

\renewcommand{\epsilon}{\varepsilon}

\newcommand{\abs}[1]{\left\vert #1\right\vert}
\newcommand{\1}[1]{{\mathbf 1}{\{#1\}}}

\newcommand{\N}{\mathbb{N}}
\newcommand{\Z}{\mathbb{Z}}
 
\newcommand{\PR}{\mathbb{P}}

\title[]{A proof of the Lyons-Pemantle-Peres monotonicity conjecture for high biases}
\date{}

\author[G.~Ben Arous]{G\'erard Ben Arous}
\address{ Courant Institute of Mathematical Sciences,
                  251 Mercer Street,
                  New York University,
                  New York, 10012-1185, U.S.A.} 
\email{gba1@nyu.edu}

\author[A.~Fribergh]{Alexander FRIBERGH}
\address{ Courant Institute of Mathematical Sciences,
                  251 Mercer Street,
                  New York University,
                  New York, 10012-1185, U.S.A.} 
\email{fribergh@cims.nyu.edu}

\author[V.~Sidoravicius]{Vladas SIDORAVICIUS}
\address{ IMPA, Estrada Dona Castorina 110, Jardim Botanico, CEP 22460-320, Rio de Janeiro, Brasil} 
\email{vladas@impa.br}

\keywords{Random walk in random environment, Galton-Watson tree} \subjclass[2000]{primary 60K37;
secondary  60D05}
\thanks{The first author was supported in part by NSF
grants DMS-0806180 and OISE-0730136. The third author was partially supported by CNPq} \subjclass[2000]{primary 60K37;
secondary 60J45, 60D05}

\begin{document}

\maketitle

\begin{abstract}
The speed $v(\beta)$ of a $\beta$-biased random walk on a Galton-Watson tree without leaves is increasing for $\beta \geq 717$.
\end{abstract}

\section{Introduction}


We study here biased random walks on Galton-Watson trees with no leaves (see~\cite{L}-\cite{Lpart}).  This paper addresses the following question: is the speed of a random walk increasing as a function of its bias?

We begin by defining the model more precisely. Consider a Galton-Watson tree without leaves. This is a random rooted tree where the offspring size of all individuals are i.i.d.~copies of an integer random variable $Z$, which verifies ${\bf P}[Z=0]=0$.  The associated probability space is denoted $(\Omega,{\bf P})$. We will use $\abs{x}$ to denote the distance of a vertex $x$ from the root. Moreover $\overleftarrow{x}$ will denote the ancestor of $x$ for any vertex $x$ different from the root.

On a Galton-Watson tree without leaves $T$, we consider the
$\beta$-biased random walk, for $\beta>0$. This is a Markov chain $(X_n)_{n\in \N}$ on the vertices of $T$, such that if $x$ is not the root and has $k$ children $u_1,\ldots, u_k$, then
\begin{enumerate}
\item $P[X_{n+1}=\overleftarrow{x}|X_n=x]=\frac 1 {1+ \beta k}$, 
\item $P[X_{n+1}= u_i |X_n=x] = \frac {\beta}{1+\beta k}$, for $1\leq i\leq k$,
\end{enumerate}
and from the root all transitions to its children are equally likely. 

We start the walk from the root of the tree and denote by $P^{\omega}[\cdot]$ 
the law of $(X_n)_{n=0,1,2, \ldots}$ on a tree $\omega$.

 We define the averaged law as the semi-direct product $\PR={\bf P} \times P^{\omega}$.

One of the results of~\cite{L} is that if $\beta>1/{\bf E}[Z]$, then the walk is transient, i.e.
\[
\lim \abs{X_n}=\infty, \qquad \PR-\text{a.s.}
\]
and we also learn from~\cite{LPP} that there exists a constant $v(\beta,{\bf P})$ depending only on $\beta$ and ${\bf P}$ such that
\[
\lim \frac{\abs{X_n}}n = v(\beta,{\bf P}), \qquad \PR-\text{a.s.},
\]
this constant $v(\beta,{\bf P})$ is called the speed of the random walk on the Galton-Watson tree.


In the rest of the paper, the dependence of the speed with respect to environment will often be omitted.

Our main result is the following
\begin{theorem}
\label{theorem}
The speed $v(\beta)$ of a $\beta$-biased random walk on a Galton-Watson tree without leaves is increasing for $\beta \geq 717$.
\end{theorem}
 
 It was conjectured in 1996 in~\cite{LPP} by Lyons, Pemantle and Peres that $v(\beta)$  in an increasing function  on the interval $(1/{\bf E}[Z],\infty)$ (see~\cite{Lpart} for a detailed discussion). The only recent progress, in~\cite{BHOZ}, gave a proof of the Einstein relation for biased random walks on Galton-Watson trees, which implies that the previous conjecture holds in a neighborhood of $1/{\bf E}[Z]$. Our proof relies on a coupling argument inspired by a tool introduced in~\cite{FG}, called super-regeneration times.

\begin{remark}
It is interesting to note that there exists an explicit expression for the limiting speed of biased random walks on Galton-Watson trees, see~\cite{Aidekon} and~\cite{GPV} for a related but different model. At this point, we do not know how to 
prove the monotonicity conjecture using these expressions.
\end{remark}
 
 Let us explain what the difficulty of this problem is since, indeed, questions about the speed of  random walks in random environments can be subtle (see~\cite{Zeitouni}, \cite{SZ2} and \cite{sznitmannew} for general reviews of the subject). On Galton-Watson trees with leaves (see~\cite{LPP} and~\cite{FG} or~\cite{GBA} and~\cite{H}), or on supercritical percolation clusters (see~\cite{BGP},~\cite{Sznitman} and ~\cite{FH}), the speed is certainly not increasing since it eventually vanishes. In these models, the slowdown of the walk can be explained by the presence of dead-ends in the environment, which act as powerful traps. The particle is, for high biases, typically spending most of its time in traps, i.e.~slow parts of the environment. 

If a tree  $T$ has no leaves, dead-ends do not exist, hence a possible slowdown cannot be explained by trapping. Nevertheless, certain parts of $T$ will be atypically thin and a biased random walk will typically go through them slower than it would in other parts of $T$. The proof of Theorem~\ref{theorem} consists in showing that, for high enough biases, this slowdown effect is not important. It must be noted that there exists several examples of trees without strong traps where this effect is strong enough to slow the walk down, see~\cite{Lpart}. 

 The methods used in our proof of Theorem~\ref{theorem} allow us to obtain more results in a simple manner. On the one hand, we can strengthen our main result when the minimal degree is not $1$.
 \begin{theorem}
\label{theorem2}
Let us consider a $\beta$-biased random walk on a Galton-Watson tree with minimal degree $d:=\min\{k\geq 1, {\bf P}[Z=k] >0\}$. The speed $v(\beta)$ of this biased random walk is increasing for $\beta \geq 717/d$.
\end{theorem}
 
 On the other hand, we can obtain information on the rate of growth of the velocity for large $\beta$.
 \begin{theorem}
\label{theorem3}
For every $\epsilon>0$, we have
\[
\frac{v(\beta+\epsilon)-v(\beta)}{\epsilon}\sim 2{\bf E}\Bigl[\frac 1Z\Bigr]\frac 1 {\beta^2},
\]
for $\beta$ large.
\end{theorem}

The previous result suggests, via a non-rigorous inversion of limits, that $ v'(\beta)\sim 2{\bf E}[1/Z]\beta^{-2}$ for $\beta$ large.

The proofs of the last two theorems will be sketched in the last section of this paper.
 
We end this introduction by a discussion of related problems and open questions. Obviously, it remains to be proved that the monotonicity conjecture holds in the entire transient regime. It seems to us, that our techniques are not sufficient for that. Although we could lower the threshold to some number lower than 717, we do not believe we could bring it down to $1$, not to mention $1/{\bf E}[Z]$. Once this question is settled, one may try to prove the following conjecture: the speed of a biased random walk on a supercritical Galton-Watson tree with leaves is unimodal. We make the same conjecture for the speed of a biased random walk on a supercritical percolation cluster on the lattice.

There are other interesting questions about monotonicity properties of the speed. We may ask, for example, monotonicity  questions with respect to the environment, more specifically, if ${\bf P}$ is stochastically dominated by ${\bf P}'$, do we have $v(\beta,{\bf P})\leq v(\beta,{\bf P}')$? A simple random walk on a percolation cluster of a regular tree  produces has a smaller speed than its counterpart on the regular tree, as follows from~\cite{Chen}. Furthermore the speed of a biased random walks on a high-density percolation cluster of $\Z^d$ is lower than that of a biased random walk on $\Z^d$, see~\cite{Fribergh}.

For the rest of the paper, we fix $\beta>1$ and $\epsilon>0$. 

\section{A coupling of three random walks}

Our aim is to explain the construction of a coupling of three different random walks, two of which are $\beta$ and a $(\beta+\epsilon)$ biased random walk on Galton-Watson trees. We ultimately wish to turn the statement of Theorem~\ref{theorem} into a question on a $\beta$ biased random walk on $\Z$ and this will be the third walk involved in this coupling.

We introduce the following notations 
\[
p_i^{(\beta)}=\frac{\beta}{i\beta+1},\ q_i^{(\beta)}=\frac 1{i\beta+1}\text{ and }\epsilon_i^{(\beta)}=q_i^{(\beta)}-q_i^{(\beta+\epsilon)}. 
\]

When trying to couple a $\beta$ and a $(\beta+\epsilon)$ biased random walks on a site with $i$ offsprings, we can make the walks stay coupled with probability $1-\epsilon_i^{(\beta)}$. Let us give the explicit coupling. We choose a sequence of $(U_i)_{i\geq 1}$ of uniform random variables in $[0,1]$  and $(Z_i)_{i\geq 0}$ with law ${\bf P}$. The probability measure associated with those random variables is denoted $P$. Using those random variables we can defined two walks $X^{(\beta)}$ and $X^{(\beta+\epsilon)}$ in the following manner: 
\begin{enumerate}
\item When the walk $X^{(\beta)}$ discovers a new site at time $k$, that site is assigned to have $Z_k$ offspring. The same procedure is used for $X^{(\beta+\epsilon)}$.
\item Assume that $X^{(\beta)}_n=x$, a site with $k$ descendants $x_1,\ldots, x_k$. If $U_{n+1}\leq q_k^{(\beta)}$, then we have $X^{(\beta)}_{n+1}=\overleftarrow{x}$ and $X^{(\beta)}_{n+1}=x_i$ if $U_{n+1}\in (1-ip_k^{(\beta)},1-(i-1) p_k^{(\beta)}]$ for any $i\in [1,k]$.
\item Assume that $X^{(\beta+\epsilon)}_n=x$, a site with $k$ descendants $x_1,\ldots, x_k$. If $U_{n+1}\in [\epsilon_k^{(\beta)},q_k^{(\beta)})$, then we have $X^{(\beta+\epsilon)}_{n+1}=\overleftarrow{x}$ and $X^{(\beta+\epsilon)}_{n+1}=x_i$ if $U_{n+1}\in (1-(i+1)p_k^{(\beta)},1-i p_k^{(\beta)}]$ for any $i\in [1,k]$. Finally if $U_{n+1}\in [(j-1)/k\epsilon_k^{(\beta)},j/k\epsilon_k^{(\beta)})$, then $X^{(\beta+\epsilon)}_{n+1}=x_j$.
\end{enumerate}

A careful reader will notice that the coupling above does not take into account  the specificity of the root. We will explain in Remark~\ref{rem_root} how the coupling works at 0 and why this is not important for the rest of the paper.

The following properties are obvious
\begin{enumerate}
\item $X^{(\beta)}$ has the law a $\beta$-biased random walk under the measure $\PR$.
\item $X^{(\beta+\epsilon)}$ has the law a $(\beta+\epsilon)$-biased random walk under the measure $\PR$.
\end{enumerate}

Using $(U_1)_{i\geq 1}$ and $(Z_i)_{i \geq 0}$, we may also define the random variables
\[
Y_n=\sum_{i=1}^{n} (\1{U_i> q_1^{(\beta)}}-\1{U_i\leq q_1^{(\beta)}}),
\]
where $Y_0=0$.

The sequence $Y=(Y_n)_{n\geq 0}$ has the law of a $\beta$-biased random walk on $\Z$.

\section{A common regeneration structure}

In this section we will construct a regeneration structure which is common to all three walks. Informally a regeneration time is a maximum of a random walk which is also a minimum of the future of the random walk. For background on regeneration times  in general we refer to~\cite{SZ} or~\cite{Zeitouni}, in the specific case of biased random walks on Galton-Watson trees the reader can consult~\cite{LPP}.   In the case of a $\beta$-biased random walk $Y_n$ on $\Z$, a time $n_0$ is a regeneration time if
\[
Y_{n_0} > \max_{n<n_0} Y_n \text{ and } Y_{n_0}< \min_{n>n_0} Y_n.
\]

The common regeneration structure is based on  the concept of super-regeneration times introduced in~\cite{FG}.

Let us introduce the notation $\{0-\text{SR}\}$, the event that $0$ is a regeneration time for $Y$. This event is measurable with respect to $\sigma(U_i,~i\geq 1)$. Its $P$-probability is $p_{\infty}=(\beta/(\beta+1))\times (\beta-1)/(\beta+1)$, the probability that a $\beta$-biased random walk on $\Z$ never returns to the origin. This allows us to introduce the measure $\tilde{P}[\,\cdot\,]=P[\,\cdot \mid 0-\text{SR}]$. 

Under $\tilde{P}$, we define the sequence of consecutive regeneration times $\tau_0=0,\tau_1,\ldots$, hence, under $\tilde{P}$, $\tau_1$ is the first non-zero regeneration time. We point out that $\{0-\text{SR}\}$ and the $\tau_i$ are expressed in terms of $Y$ and, as such, have a law that depends only on the parameter $\beta$.

The key observation is the following: if $\tau_1$ is a regeneration time for $Y$, then $\tau_1$ is a regeneration time for $X^{(\beta)}$ and $X^{(\beta+\epsilon)}$, in the sense usually employed on trees (see~\cite{LPP}). In this context, we choose to call $\tau_1$  a super-regeneration time.

Using classical arguments from the theory of regeneration times, we may see that, under $P$, the sequence $(X_{\tau_{i+1}}-X_{\tau_i},\tau_{i+1}-\tau_i)_{i\geq 1}$ is i.i.d.~and has the same law as $(X_{\tau_1},\tau_1)$ under $\tilde{P}$.  This leads to the following proposition (where we recall that $E[\tau_1]$ and $\tilde{E}[\tau_1]$ are obviously finite). 
\begin{proposition}
\label{prop_speed1}
For any $\beta>1$ and $\epsilon>0$. We have 
\[
v(\beta)=\frac{\tilde{E}\Bigl[\abs{X_{\tau_1}^{(\beta)}}\Bigr]}{\tilde{E}[\tau_1]} \text{ and } v(\beta+\epsilon)=\frac{\tilde{E}\Bigl[\abs{X_{\tau_1}^{(\beta+\epsilon)}}\Bigr]}{\tilde{E}[\tau_1]}.
\]

In particular, if $\tilde{E}\Bigl[\abs{X_{\tau_1}^{(\beta+\epsilon)}}-\abs{X_{\tau_1}^{(\beta)}}\Bigr]>0$, then $v(\beta+\epsilon)>v(\beta)$.
\end{proposition}

All the details needed to prove this proposition are standard and we refer the reader who wants to see them to~\cite{LPP}.

\begin{remark} \label{rem_root}
The previous proposition implies that as far as the speed is concerned we only need to understand the walks under the conditioning that $0-\text{SR}$. In this case, the walks visits the root (which has $Z_1$ descendants, $x_1,\ldots, x_{Z_1}$) only once and $Y$ necessarily goes downwards (i.e.~$U_1> q_1^{(\beta)}$). The way the coupling works in that case is as follows: $U_{1}\in (1-i (1-q_1^{(\beta)})/Z_1,1-(i-1) (1-q_1^{(\beta)})/Z_1]$ for any $i\in [1,Z_1]$.
\end{remark}

\section{Understanding the coupling}

 First, we introduce the event that the walks remained \lq\lq together\rq\rq for the entire regeneration period:
\[
C=\Bigl\{\text{for all $i\leq \tau_1$, }\abs{X_i^{(\beta)}}=\abs{X_i^{(\beta+\epsilon)}}\Bigr\}.
\]

If they do not stay together, we will say that the walks decouple. We will distinguish different ways of decoupling, for any $k\geq 1$, we introduce the event of decoupling with $k$ returns:
\[
D_k=C^c\cap \{\abs{\mathfrak{B}}=k\},
\]
where $\mathfrak{B}$ is the set of times before $\tau_1$ when $Y$ takes a step back, or equivalently $\{i\in[1, \tau_1],~U_i\leq q_1^{(\beta)}\}$.

Before moving forward, let us point out a simple relationship between $\mathfrak{B}$ and $\tau_1$, which will turn out pivotal for the rest of the proof. 
 \begin{lemma}\label{rem_b}
If $\{\abs{\mathfrak{B}}=k\}$, then $\{\tau_1\leq 3k+1\}$. 
\end{lemma}
\begin{remark} The definition of $\tau_1$ is slightly ambiguous, under $\tilde{P}$ it cannot be $0$ whereas this is possible under the measure $p$. In this particular lemma, $\tau_1$ can take the value $0$.
\end{remark}
\begin{proof}
On $\{\abs{\mathfrak{B}}=k\}$, we may notice that $Y_{\tau_1}\leq k+1$. Indeed, by the pigeon hole principle there exists $j\leq k+1$ such that $Y$ does not go back from $j$ to $j-1$, before $\tau_1$. But, this actually means that the hitting time of $j$ is actually a regeneration time, which contradicts the definition of $\tau_1$.

On $\{\abs{\mathfrak{B}}=k\}$, we know that $Y_{\tau_1}\leq k+1$ and since $Y_{\tau_1}\geq Y_i$ for any $i\leq \tau_1$, we obtain that for any $i\leq \tau_1$, we have $Y_i\leq k+1$. Moreover, on $\{\abs{\mathfrak{B}}=k\}$, we see that for any $i\leq \tau_1$ we have $Y_i\geq i-2k$. This means that for $i\leq \tau_1$, we have $i\leq 3k+1$,  which implies that $\tau_1\leq 3k+1$.
\end{proof}

On the event $C^c$, we introduce the decoupling time $\delta=\inf\Bigl\{i\leq \tau_1,~\abs{X_i^{(\beta)}}\neq \abs{X_i^{(\beta+\epsilon)}}\Bigr\}$. We necessarily have $\delta \in \mathfrak{B}$. Hence $C$ and $(D_k)_{k\geq 1}$ form a partition of the space $\Omega$.

\begin{remark} Let us make three key observations
\begin{enumerate}
\item $\abs{X_{\delta}^{(\beta+\epsilon)}}-\abs{X_{\delta}^{(\beta)}}=2$, meaning that when the walks decouple, the more biased one is always the one moving forward whereas the other is the one moving back. 
\item if $n\notin \mathfrak{B}$ and $n\leq \tau_1$, then $X_{n+1}^{(\beta)}$ (resp.~$X_{n+1}^{(\beta+\epsilon)}$) has to go to a descendant of $X_{n}^{(\beta)}$ (resp.~$X_n^{(\beta+\epsilon)}$). This means that for $n\notin \mathfrak{B}$ and $n\leq \tau_1$, we have $\abs{X_{n}^{(\beta+\epsilon)}}-\abs{X_{n}^{(\beta)}}=0$.
\item if $n\leq \tau_1$ and $n\in \mathfrak{B}\setminus \{\delta\}$, then $\abs{X_{n+1}^{(\beta+\epsilon)}}-\abs{X_{n+1}^{(\beta)}}\geq -2$.
\end{enumerate}
\end{remark}

From this remark, we may notice the following obvious statements
\begin{enumerate}
\item on $C$, we have $\abs{X^{(\beta+\epsilon)}_{\tau_1}}=\abs{X^{(\beta)}_{\tau_1}}$,
\item on $D_k$, we have $\abs{X^{(\beta+\epsilon)}_{\tau_1}}\geq \abs{X^{(\beta)}_{\tau_1}}+2-2(k-1)$, for any $k\geq 1$.
\end{enumerate}

Hence, we have
\begin{equation}\label{blabla}
\tilde{E}\bigl[\abs{X_{\tau_1}^{(\beta+\epsilon)}}-\abs{X_{\tau_1}^{(\beta)}}\bigr] 
\geq  2\Bigl[\tilde{P}[D_1]-\sum_{k\geq 2} (k-2) \tilde{P}[D_k]\Bigr] .
\end{equation}

This last equation and Proposition~\ref{prop_speed1} imply that
\begin{proposition}
\label{prop_speed2}
If $\tilde{P}[D_1]>\sum_{k\geq 2} (k-2)\tilde{P}[D_k]$, then $v(\beta+\epsilon)>v(\beta)$.
\end{proposition}

 \section{Reducing the problem to a question about biased random walks on $\Z$}
 
 In this section, we will give a lower bound on $ \tilde{P}[D_1]$ and an upper bound on $ \tilde{P}[D_k]$ for $k\geq 2$.
 
 \begin{lemma}\label{lem1}
 We have
 \[
 \tilde{P}[D_1] \geq (p_1^{(\beta)})^4  {\bf E}[\epsilon_Z^{(\beta)}].
 \]
 \end{lemma}
 
  \begin{proof}
  We can use the following scenario to get an event which is in $D_1\cap \{0-\text{SR}\}$: 
  \begin{enumerate}
  \item Y makes two steps forward (this happens with $P$-probability $(p_1^{(\beta)})^2$).
  \item At time 3, $X^{(\beta)}$ and $X^{(\beta+\epsilon)}$ decouple (this happens with $P$-probability ${\bf E}[\epsilon_Z^{(\beta)}]$).
  \item Then, Y goes twice forward and at time 5, Y has a regeneration (this happens with $P$-probability $(p_1^{(\beta)})^2p_{\infty}$).
  \end{enumerate}
  
 Hence
 \[
 \tilde{P}[D_1]= \frac 1{p_{\infty}} P[D_1\cap \{0-\text{SR}\}] \geq (p_1^{(\beta)})^4 {\bf E}[\epsilon_Z^{(\beta)}].
 \]
 \end{proof}
 
We will now prove an upper-bound on $\tilde{P}[D_k]$.
 \begin{lemma}\label{lem2}
 For any $k\geq 2$, we have
 \[
 \tilde{P}[D_k]\leq \Bigl( (q_1^{(\beta)})^{-1}p_{\infty}^{-1}{\bf E}[\epsilon_Z^{(\beta)}] \Bigr) k(3k+1) P[\abs{\mathfrak{B}}=k].
 \]
 \end{lemma}
 
 \begin{proof}
 We may see that
 \begin{align}
 \label{step1}
  \tilde{P}[D_k] & \leq p_{\infty}^{-1} P[D_k] \\ \nonumber
                          & = p_{\infty}^{-1}\sum_{n\geq 0} \sum_{ (u_1,\ldots,u_k) \subset [0,n]} \sum_{j=1}^k P[\tau_1=n, \{u_1,\ldots,u_k\}=\mathfrak{B}, \delta=u_j].
 \end{align}
 
 We may notice that $\tau_1$ and $\mathfrak{B}$ are actually measurable with respect to $Y$ (i.e.~the sequence $(\1{U_i\leq q_1^{(\beta)}})_{i\geq 0}$).
 
  For all $i_0\geq 0$, we introduce the random variables $\tau_1^{(i_0)}$ (resp.~$\mathfrak{B}^{(i_0)}$) which are equal to $\tau_1$ (resp.~$\mathfrak{B}$) conditioned on $\{U_{i_0}\leq q_1^{(\beta)}\}$. Those random variables are $(\1{U_i\leq q_1^{(\beta)}})_{i\neq i_0}$-measurable and coincide with $\tau_1$ (resp.~$\mathfrak{B}$) when $\{U_{i_0}\leq q_1^{(\beta)}\}$.
  
  On the event $\{\delta=u_j\}$, we have $U_{u_j}\leq q_1^{(\beta)}$. Hence, on the event $\{\delta=u_j\}$, we have $\tau_1=\tau_1^{(u_j)}$ and $\mathfrak{B}=\mathfrak{B}^{(u_j)}$. This means that 
 \[
 P[\tau_1=n, \{u_1,\ldots,u_k\}=\mathfrak{B}, \delta=u_j]= P\bigl[\tau_1^{(u_j)}=n, \{u_1,\ldots,u_k\}=\mathfrak{B}^{(u_j)}, \delta=u_j\bigr].
 \]
 
 Now, we may see that $\{\delta=u_j\}$ implies that $U_{u_j}\leq \max_{1\leq l\leq n}\epsilon_{Z_{l}}^{(\beta)}$. Thus
                          \begin{align*}
                          & P\bigl[\tau_1^{(u_j)}=n, \{u_1,\ldots,u_k\}=\mathfrak{B}^{(u_j)}, \delta=u_j\bigr]  \\
                          \leq &P\bigl[\tau_1^{(u_j)}=n, \{u_1,\ldots,u_k\}=\mathfrak{B}^{(u_j)}, U_{u_j}\leq \max_{1\leq l\leq n}\epsilon_{Z_{l}}^{(\beta)}\bigr],
                          \end{align*}
                          and using the fact that $\{\tau_1^{(u_j)}=n, \{u_1,\ldots,u_k\}=\mathfrak{B}^{(u_j)}\}$ and $\{U_{u_j}\leq \max_{1\leq l\leq n}\epsilon_{Z_{l}}^{(\beta)}\}$ are $P$-independent, we obtain (putting the two previous equations together)
                          \begin{align}\label{step3}
 & P[\tau_1=n, \{u_1,\ldots,u_k\}=\mathfrak{B}, \delta=u_j] \\ \nonumber 
  \leq & P[U_{u_j}
 \leq \max_{1\leq l\leq n}\epsilon_{Z_{l}}^{(\beta)}] P[\tau_1^{(u_j)}=n, \{u_1,\ldots,u_k\}=\mathfrak{B}^{(u_j)}].
 \end{align}
 
 Since $\{\tau_1^{(u_j)}=n, \{u_1,\ldots,u_k\}=\mathfrak{B}^{(u_j)}\}$ is $P$-independent of $U_{u_j}$, we can make the following transformation
 \begin{align*}
  &P[\tau_1^{(u_j)}=n, \{u_1,\ldots,u_k\}=\mathfrak{B}^{(u_j)}] \\
 =& \frac 1 {P[U_{u_j}\leq q_1^{(\beta)}]} P[\tau_1^{(u_j)}=n, \{u_1,\ldots,u_k\}=\mathfrak{B}^{(u_j)}, U_{u_j}\leq q_1^{(\beta)}] \\
 =&  (q_1^{(\beta)})^{-1} P[\tau_1=n, \{u_1,\ldots,u_k\}=\mathfrak{B}],
 \end{align*}
 where we simply used the definition of $\tau_1^{(u_j)}$ and $\mathfrak{B}^{(u_j)}$ to obtain the last line. Moreover, we have that $P[U_{u_j}\leq \max_{1\leq l\leq n}\epsilon_{Z_{l}}^{(\beta)}]={\bf E}[\max_{1\leq l\leq n}\epsilon_{Z_{l}}^{(\beta)}]\leq n {\bf E}[\epsilon_Z^{(\beta)}]$. Now, using Lemma~\ref{rem_b}, we may notice that if $\abs{\mathfrak{B}}\leq k$ then $\tau_1\leq 3k+1$, so the positive terms in the sum in~(\ref{step1}) verify $n\leq 3k+1$. This means that~(\ref{step3}) can be written
 \begin{align*}
 & P[\tau_1=n, \{u_1,\ldots,u_k\}=\mathfrak{B}, \delta=u_j]\\
 \leq & (3k+1){\bf E}[\epsilon_Z^{(\beta)}]
  (q_1^{(\beta)})^{-1} P[\tau_1=n, \{u_1,\ldots,u_k\}=\mathfrak{B}].
  \end{align*}
  
  Inputing this information in~(\ref{step1}) means that
  \begin{align*}
  \tilde{P}[D_k] &\leq  (q_1^{(\beta)})^{-1}p_{\infty}^{-1} {\bf E}[\epsilon_Z^{(\beta)}]
(3k+1)k  \sum_{n\geq 0} \sum_{ (u_1,\ldots,u_k) \subset [0,n]}  P[\tau_1=n, \{u_1,\ldots,u_k\}=\mathfrak{B}]\\
   &\leq  (q_1^{(\beta)})^{-1}p_{\infty}^{-1} {\bf E}[\epsilon_Z^{(\beta)}](3k+1)k P[\abs{\mathfrak{B}}=k],
   \end{align*}
   which finishes the proof of the lemma.
  \end{proof}
 
 Using Lemma~\ref{lem1}, Lemma~\ref{lem2} and Proposition~\ref{prop_speed2}, a simple computation yields the following.
 \begin{proposition}
 \label{prop_speed3}
 If  $(p_1^{(\beta)})^{-4} p_{\infty}^{-1} (q_1^{(\beta)})^{-1} \sum_{k\geq 2}(3k+1)k(k-2)P[\abs{\mathfrak{B}}=k] <1$, then $v(\beta+\epsilon)>v(\beta)$.
 \end{proposition}
 
  We may notice that the random variable $\abs{\mathfrak{B}}$ depends only on the walk $Y$ and our remaining task is a question about a $\beta$-biased random walk on $\Z$.

 \section{Estimate on a $\beta$-biased random walk on $\Z$.}
 
 In this section, we will give an upper bound on $P[\abs{\mathfrak{B}}=k]$ for $k\geq 2$. Our method will not yield an optimal result but has the advantage of being short and simple. We can of course obtain a better upper bound on $P[\abs{\mathfrak{B}}=k]$, but we believe that this would only allow us to obtain Theorem~\ref{theorem} for $\beta\geq \beta_0$ for some $\beta_0>1$. Conceptually, this does not appear to make a very big difference and we prefer to go for simplicity. We have
 \begin{lemma}
 \label{tail_b}
 We have
 \[
 P[\abs{\mathfrak{B}}=k] \leq \Bigl(\frac{27q_1^{(\beta)}}{4 }\Bigr)^k.
 \]
 \end{lemma}
 \begin{proof}
 On $\{\abs{\mathfrak{B}}=k\}$, Lemma~\ref{rem_b} implies that in the first $3k+1$ steps, $Y$ takes at least $k$ steps back. Hence
 \begin{align*}
  P[\abs{\mathfrak{B}}=k]  & \leq {3k+1 \choose k} (q_1^{(\beta)})^k \\
 &  \leq \frac{3k+1}{2k+1} {3k \choose k} (q_1^{(\beta)})^k 
  \leq \frac 32  \frac{e}{e^{22/12}} \Bigl(\frac{27 q_1^{(\beta)}}{4}\Bigr)^k,
  \end{align*}
  where we used the inequality $e^{11/12} (n/e)^n < n! <e (n/e)^n$ which can be found in~\cite{Hummel}. The result follows by evaluating the constant.
 \end{proof}

Using the previous lemma, we may see that
\begin{align*}
 \sum_{k\geq 2}(3k+1)k(k-2)P[\abs{\mathfrak{B}}=k] & \leq 5 \sum_{k\geq 2} k(k-1)(k-2) \Bigl(\frac{27 q_1^{(\beta)}}{4}\Bigr)^k \\
 & =15 \Bigl(\frac{27 q_1^{(\beta)}}{4}\Bigr)^2 \Bigl(1-\frac{27 q_1^{(\beta)}}{4}\Bigr)^{-4}.
 \end{align*}
 
By Proposition~\ref{prop_speed3}, we may see that $v(\beta+\epsilon)>v(\beta)$, if $C(\beta)<1$, where
 \[
C(\beta):=15(p_1^{(\beta)})^{-4} p_{\infty}^{-1} (q_1^{(\beta)})^{-1} \Bigl(\frac{27 q_1^{(\beta)}}{4}\Bigr)^2 \Bigl(1-\frac{27 q_1^{(\beta)}}{4}\Bigr)^{-4}.
\]

We may see that $C(\beta)\sim 15(27/4)^2 \beta^{-1}$, so the previous equation will be satisfied for $\beta$ large enough. We may actually see that $C(\beta)<1$, for any $\beta\geq 717$ and this proves Theorem~\ref{theorem}.

 \section{Extensions of the main result}
 
 In this section, we prove Theorem~\ref{theorem2} and Theorem~\ref{theorem3}.
 
 \subsection{When the minimal degree is larger than one}
 
 If the minimal degree of the tree $d:=\min\{k\geq 1, {\bf P}[Z=k] >0\}$ verifies $d\geq 2$, we can easily extend our main result to obtain Theorem~\ref{theorem2}. Indeed, it is sufficient to consider a $d\beta$-biased random walk on $\Z$ instead of simply considering a $\beta$-biased one. Formally, it would be sufficient to use the following notations in the previous proof
 \[
 p_1^{(\beta)}=\frac{d\beta}{d\beta+1},\ q_1^{(\beta)}=\frac 1{d\beta+1} \text{ and } p_{\infty}= \frac{d\beta}{d\beta+1}\frac{d\beta-1}{d\beta+1}.
 \]
 
 We emphasize that $\beta$ could possibly be lower than one, as long as $d\beta>717$.
 
 \subsection{Asymptotic rate of increase for the speed}
 
 Let us sketch how to obtain Theorem~\ref{theorem3}. Using Proposition~\ref{prop_speed1} and the fact that $\tilde{E}[\tau_1]=1+o(1)$, we obtain
 \[
 v(\beta+\epsilon)-v(\beta)\sim\tilde{E}\Bigl[\abs{X_{\tau_1}^{(\beta+\epsilon)}}-\abs{X_{\tau_1}^{(\beta)}}\Bigr],
 \]
 which by~(\ref{blabla}) gives
 \[
 v(\beta+\epsilon)-v(\beta)\sim 2\Bigl[\tilde{P}[D_1]-\sum_{k\geq 2} (k-2) \tilde{P}[D_k]\Bigr].
 \]
 
 Using Lemma~\ref{lem1} and Lemma~\ref{lem2}, we can see that
 \[
 v(\beta+\epsilon)-v(\beta)\sim  2 \tilde{P}[D_1].
 \]
 
 The reader will see that the scenario described in Lemma~\ref{lem1} is the only possibility for $D_1$ under the measure $\tilde{P}$, so it can easily be seen that $\tilde{P}[D_1]\sim {\bf E}[\epsilon_Z^{(\beta)}] \sim \epsilon \beta^{-2} {\bf E}[1/Z]$. This means that 
 \[
 v(\beta+\epsilon)-v(\beta)\sim  2 \epsilon {\bf E}\Bigl[\frac 1Z\Bigr] \frac 1{\beta^2},
 \]
  so that Theorem~\ref{theorem3} follows.

\noindent{\bf Acknowledgments.}
Vladas Sidoravicius would like to thank the Courant Institute of Mathematical Science for its hospitality and financial support. G\'erard Ben Arous would like to thank IMPA for its hospitality. The three authors are grateful to Nina Gantert and Amir Dembo for reading an early version of this paper.

\end{document}